\documentclass[12pt,oneside]{article}
\usepackage[english]{babel}

\usepackage[letterpaper,top=2cm,bottom=2cm,left=3cm,right=3cm,marginparwidth=1.75cm]{geometry}

\usepackage{amsmath}
\usepackage{amsthm}
\usepackage{amssymb}
\usepackage{times}
\newtheorem{theorem}{Theorem}
\newtheorem{lemma}{Lemma}
\newtheorem{corollary}{Corollary}
\newtheorem{proposition}{Proposition}
\theoremstyle{definition}
\newtheorem{definition}{Definition}
\newtheorem{remark}{Remark}

\newtheorem{conjecture}{Conjecture}
\usepackage{graphicx}
\usepackage[colorlinks=true, allcolors=blue]{hyperref}
\usepackage{mathtools}
\usepackage{indentfirst}
\usepackage{xcolor}  
\usepackage{array}
\usepackage{tikz}
\usepackage{multicol}
\usetikzlibrary{decorations.pathreplacing}
\DeclareMathOperator{\dist}{dist}
\DeclareMathOperator{\diag}{diag}
\DeclareMathOperator{\trs}{trs}
\DeclareMathOperator{\Trs}{Trs}
\DeclareMathOperator{\Deg}{Deg}
\DeclareMathOperator{\SNF}{SNF}

\newcommand{\Ddeg}{D^{\operatorname{deg}}}
\newcommand{\Ddegp}{D^{\operatorname{deg},+}}
\newcommand{\Atrs}{A^{\operatorname{trs}}}
\newcommand{\Atrsp}{A^{\operatorname{trs},+}}
\newcommand{\ZZ}{\mathbb Z}

\newcommand{\RR}{\mathbb R}

\begin{document}
\title{A Family of Simultaneously Cospectral Trees for Degree-Distance Matrices}
\author{\small Limeng Lin$\,^{\rm a}$\quad\quad Quanyu Tang$\,^{\rm a}$\quad\quad Kehua Wang$\,^{\rm b}$ \quad\quad Wei Wang$\,^{\rm a}$\footnote{Corresponding author. Email address: wang\_weiw@163.com}\\
\small $^{\rm a}\,$School of Mathematics and Statistics, Xi'an Jiaotong University, Xi'an, 710049, P. R. China\\
\small $^{\rm b}\,$Department of Computer Science, Columbia University, New York, NY, 10027, USA}
\date{}
	\maketitle

\begin{abstract}

Spectral characterization of graphs for various graph matrices constitutes a central topic in spectral graph theory. Let $G$ be a graph with adjacency matrix $A(G)$, diagonal degree matrix $\Deg(G)$, distance matrix $D(G)$, and transmission matrix \(\Trs(G)\), respectively. Recently, Alfaro and Zapata (2024) introduced the degree-distance matrices \(\Ddegp(G)=\Deg(G)+D(G)\) and \(\Ddeg(G)=\Deg(G)-D(G)\), together with the transmission-adjacency matrices \(\Atrsp(G)=\Trs(G)+A(G)\) and \(\Atrs(G)=\Trs(G)-A(G)\). Based on computational evidence for trees on at most \(20\) vertices, they conjectured that all trees are determined by the spectra of \(\Ddegp\) as well as \(\Ddeg\).
 
 In this paper, we disprove these conjectures by constructing an infinite family of pairs of non-isomorphic trees. More precisely, for each integer \(r\ge 3\), we construct a pair of trees on \(17r-15\) vertices which are simultaneously cospectral with respect to the following six matrices
\[
        A,\quad L,\quad Q,\quad D,\quad \Ddegp,\quad \Ddeg .
\]
The construction is based on an \(r\)-regularized leaf extension and an equitable-partition reduction. We also record a simple sign-switching observation for transmission-adjacency matrices: if \(G\) is bipartite, then \(\Atrs(G)\) and \(\Atrsp(G)\) are similar via a diagonal \(\{\pm1\}\)-matrix and have the same Smith normal form. Consequently, for trees, the spectral and Smith normal form problems for \(\Atrs\) and \(\Atrsp\) are equivalent.
\end{abstract}



\section{Introduction}
Let \(G\) be a simple connected graph with vertex set
\(V(G)=\{v_1,\ldots,v_n\}\). The \emph{adjacency matrix} of \(G\) is denoted by
\(A(G)\), and the \emph{diagonal degree matrix} is denoted by \(\Deg(G)\). The
\emph{Laplacian matrix} and the \emph{signless Laplacian matrix} of \(G\) are $ L(G)=\Deg(G)-A(G)$ and $Q(G)=\Deg(G)+A(G),$
respectively. The \textit{distance matrix} of $G$ is a matrix $D(G)=\left [ \dist_G(v_i,v_j) \right ]_{n\times n}$, where $\dist_G(v_i,v_j)$ denotes shortest-path distance between $v_i$ and $v_j$.
The transmission of a vertex \(v\) is
\[\trs_G(v)=\sum_{u\in V(G)}\dist_G(u,v),\]
and the then \emph{transmission matrix}  is
\[\Trs(G)=\diag(\trs_G(v_1),\ldots,\trs_G(v_n)).\]
In~\cite{AlfaroZapata2024},  Alfaro and Zapata defined some matrices as follows:
\begin{itemize}
        \item the \emph{degree-distance matrix} $\Ddeg(G)$ of $G$ as $\Deg(G)- D(G)$,
        \item the \emph{signless degree-distance matrix} $\Ddegp(G)$ of $G$ as $\Deg(G)+D(G)$,
        \item the \emph{transmission-adjacency matrix} $\Atrs(G)$ of $G$ as $\Trs(G)-A(G)$,
        \item the \emph{signless transmission-adjacency matrix }$\Atrsp(G)$ of G as $\Trs(G)+A(G)$.
\end{itemize}

For a matrix \(M\) associated with a graph $G$, the \textit{M-spectrum} of $G$ consists of all the eigenvalues (including the algebraic multiplicities) of $M$, and two graphs \(G\) and \(H\) are called
\emph{\(M\)-cospectral} if \(G\) and \(H\) have the same $M$-spectrum.  A graph $G$ \textit{is determined by M-spectrum} if any graph $M$-cospectral with $G$ is isomorphic to $G$ ($\mathrm{DMS}$ for short).

A fundamental problem in spectral graph theory is to determine which graphs are characterized by their spectra. It was once believed that every graph might be determined by its spectrum, until Collatz and Sinogowitz~\cite{CS} constructed a pair of cospectral non-isomorphic trees. We refer to van Dam and Haemers~\cite{DH,DH1} for surveys of this problem for several graph matrices which discussed both spectral characterizations and constructions of non-isomorphic cospectral graphs. Although many results have been obtained, most known DS graphs rely on strong structural properties. Further, Wang ~\cite{W4,W3} studied this problem from the perspective of the generalized spectrum and established arithmetic criteria for determining whether a family of graphs are determined by their generalized spectrum.

Recently, Alfaro and Zapata~\cite{AlfaroZapata2024} investigated whether graphs can be distinguished by combinations of spectra of known graph-associated matrices. They introduced the degree-distance matrices and the transmission-adjacency matrices, and carried out a computational study of \(\Ddegp,\Ddeg,\Atrsp\), and \(\Atrs\) on connected graphs and trees. In particular, they reported that among all trees on at most 20 vertices, there are no \(\Ddegp\)-cospectral pairs and no \(\Ddeg\)-cospectral pairs. This led to the following Conjecture~\ref{conj}. In a related direction, Abiad et al.~\cite{AbiadAlfaroVillagran2025} computed Smith normal forms for transmission-adjacency matrices in several cases.

\begin{conjecture}[Alfaro--Zapata~\cite{AlfaroZapata2024}]\label{conj}
All trees are determined by the \(\Ddegp\)-spectrum, and all trees are determined by the \(\Ddeg\)-spectrum.
\end{conjecture}

The main result of this paper disproves the aforementioned conjecture.
More precisely, we
construct an infinite family of pairs of non-isomorphic trees that are not only
\(\Ddegp\)-cospectral and \(\Ddeg\)-cospectral, but also simultaneously
cospectral with respect to several classical graph matrices.

\begin{theorem}\label{thm:intro-main}
For each integer \(r \ge 3\), there exist two non-isomorphic trees
\(T_1^{(r)}\) and \(T_2^{(r)}\), each on \(17r - 15\) vertices, such that
\(T_1^{(r)}\) and \(T_2^{(r)}\) are simultaneously cospectral with respect to
\[
        A,\quad L,\quad Q,\quad D,\quad \Ddegp,\quad \Ddeg .
\]
\end{theorem}

\begin{remark}
Consequently, trees are generally not determined by either the $\Ddegp$-spectrum or the $\Ddeg$-spectrum. Moreover, Theorem~\ref{thm:intro-main} shows that no tuple
of spectra chosen only from the matrices in the displayed
list can determine all trees.
\end{remark}

The proof of Theorem~\ref{thm:intro-main} is based on a regularized leaf
extension. Starting from two base trees on \(17\) vertices, we
attach leaves to every original vertex until all original vertices have degree
\(r\). The resulting trees admit natural equitable partitions. The characteristic
polynomials of the six matrices above decomposes into a fixed factor
and the characteristic polynomial of the corresponding quotient matrix. The
construction is completed by verifying that the relevant quotient characteristic
polynomials are identical. More details are provided in Section~\ref{sec:main-proof}.

We also give a second observation concerning the two transmission-adjacency
matrices. If \(G\) is bipartite, then \(\Atrs(G)\) and \(\Atrsp(G)\) are similar
by a diagonal sign matrix. Since this sign matrix is unimodular, the two
matrices also have the same Smith normal form. Thus, for trees, the spectral and
Smith normal form problems for \(\Atrs\) and \(\Atrsp\) are equivalent. In addition, we prove that the infinite family of tree pairs constructed in Theorem~\ref{thm:intro-main} fails to be cospectral with respect to either of these two transmission-adjacency matrices.

The rest of the paper is organized as follows. In
Section~\ref{sec:sign-switching}, we prove the sign-switching result for
transmission-adjacency matrices. In Section~\ref{sec:regularized}, we introduce
the regularized leaf extension and derive the quotient factorizations needed in
the proof. Section~\ref{sec:main-proof} is devoted to the construction of the
cospectral trees and the proof of Theorem~\ref{thm:intro-main}.
\section{A Sign Switch for the Transmission-Adjacency Matrices}
\label{sec:sign-switching}

We first introduce a simple but useful observation on the relation between the
transmission-adjacency matrix and the signless transmission-adjacency matrix
on bipartite graphs. 

Recall that a matrix $U$ with integer entries is called \emph{unimodular} if $\det U = \pm 1$. For every integral matrix $M$ with full rank, there exist two unimodular matrices $U$ and $V$ such that $M=USV$, where $S={\rm diag}(d_1,d_2,\ldots,d_n)$, and $d_i\mid d_{i+1}$ for $1\leq i\leq n-1$. Moreover, $S$ is known as the \emph{Smith normal form} of $M$ (SNF for short).

Although the two matrices are different in general, they
become equivalent under a diagonal sign switching whenever the underlying graph
is bipartite.  In particular, for trees, the spectral and Smith normal form
problems for \(\Atrs\) and \(\Atrsp\) are identical.

\begin{lemma}\label{lem:sign-switch}
Let \(G\) be a connected bipartite graph. Then \(\Atrs(G)\) and
\(\Atrsp(G)\) are similar over \(\ZZ\). More precisely, there exists a
diagonal matrix \(S\) with diagonal entries in \(\{\pm1\}\) such that
\[
        S\Atrsp(G)S=\Atrs(G).
\]
Moreover, they have the same characteristic polynomial and the same Smith
normal form:
\[
\chi_{\Atrsp(G)}(x)=\chi_{\Atrs(G)}(x)\qquad\SNF(\Atrsp(G))=\SNF(\Atrs(G)).
\]
\end{lemma}

\begin{proof}
Let \(V(G)=X\sqcup Y\) be a bipartition of \(G\), where
\(|X|=s\) and \(|Y|=t\). Suppose
that the adjacency matrices of \(G\) has the form
\[
        A(G)=
        \begin{bmatrix}
        O & M\\
        M^{\mathrm T} & O
        \end{bmatrix},
\]
where \(M\) is an \(s\times t\) matrix. We define \(S_{vv}=1\) for \(v\in X\), and \(S_{vv}=-1\) for
\(v\in Y\), and hence
\[
        S=
        \begin{bmatrix}
        I_s & O\\
        O & -I_t
        \end{bmatrix}.
\]
 Since \(\Trs(G)\) is diagonal, we have $S\Trs(G)S=\Trs(G).$
On the other hand,
\[
\begin{aligned}
        SA(G)S
        &=
        \begin{bmatrix}
        I_s & O\\
        O & -I_t
        \end{bmatrix}
        \begin{bmatrix}
        O & M\\
        M^{\mathrm T} & O
        \end{bmatrix}
        \begin{bmatrix}
        I_s & O\\
        O & -I_t
        \end{bmatrix}  \\
        &=
        \begin{bmatrix}
        O & -M\\
        -M^{\mathrm T} & O
        \end{bmatrix}
        =
        -A(G).
\end{aligned}
\]
Therefore, $S\bigl(\Trs(G)+A(G)\bigr)S=\Trs(G)-A(G)$ ,
that is, $S\Atrsp(G)S=\Atrs(G).$

Since \(S^{-1}=S\), the two matrices are similar over \(\ZZ\), and hence their
characteristic polynomials agree. Finally, \(S\) is unimodular over \(\ZZ\). Thus left and right multiplication
by \(S\) gives an equivalence over \(\ZZ\), which preserves the Smith normal
form.
\end{proof}

\begin{corollary}\label{cor:trees-plus-minus-Atrs}
For every tree \(T\),
\[
        \chi_{\Atrsp(T)}(x)=\chi_{\Atrs(T)}(x),
        \qquad
        \SNF(\Atrsp(T))=\SNF(\Atrs(T)).
\]
\end{corollary}

\begin{proof}
Every tree is bipartite. The result follows immediately from
Lemma~\ref{lem:sign-switch}.
\end{proof}

\begin{remark}
Thus, in the class of trees, the spectral and Smith normal form problems for
\(\Atrs\) and \(\Atrsp\) should be treated as a single problem. In particular,
two trees are \(\Atrs\)-cospectral if and only if they are
\(\Atrsp\)-cospectral.
\end{remark}

\begin{remark}
Alfaro and Zapata reported that there are
no \(\Atrs\)-cospectral trees with up to \(20\) vertices, while they
exhibit two \(\Atrsp\)-cospectral trees on \(20\) vertices; see ~\cite[Section~2.1]{AlfaroZapata2024}. Since every tree is
bipartite, Corollary~\ref{cor:trees-plus-minus-Atrs} implies that these two
statements cannot both hold as stated. In particular, any
\(\Atrsp\)-cospectral pair of trees is automatically an \(\Atrs\)-cospectral
pair.
\end{remark}

\section{Regularized Leaf Extensions of Trees}
\label{sec:regularized}


Let \(G\) be a tree with vertex set \(V(G)=\{v_1,\ldots,v_n\}\), and let
\(\Delta(G)\) denote its maximum degree. Fix an integer \(r\ge \Delta(G)\), and
put $p_i=r-\deg_G(v_i),\  1\le i\le n$.
\begin{definition}
The \emph{\(r\)-regularized leaf extension} of \(G\), denoted by \(R_r(G)\), is
the tree obtained from \(G\) by attaching \(p_i\) new leaves to \(v_i\), for each
\(1\le i\le n\).
\end{definition}

Thus, every original vertex of \(G\) has degree \(r\) in \(R_r(G)\), whereas all
new vertices are leaves. Let $I_+=\{i:p_i>0\}.$
The tree \(R_r(G)\) has a natural equitable partition
\begin{equation}
        \Pi=\{O_1,\ldots,O_n\}\cup\{L_i:i\in I_+\},
        \label{pi}
\end{equation}
where \(O_i=\{v_i\}\) is the singleton consisting of the original vertex \(v_i\),
and \(L_i\) is the set of leaves attached to \(v_i\). For convenience, we set $\mathcal{M}=\{A,L,Q,D,\Ddegp,\Ddeg\},$  \(A_G=A(G)\), \(D_G=D(G)\), \(P=\operatorname{diag}(p_1,\ldots,p_n)\). Let $I$ be $n$
-order identity matrix and
\(J\) is the all-one matrix of order \(n\).

\begin{proposition}\label{prop:quotient-matrices}
Let \(G\) be a tree, and let \(R_r(G)\) be its \(r\)-regularized leaf extension. Suppose that \(I_+=\{1,\ldots,n\}\).
With respect to the equitable partition \(\Pi\) defined in \eqref{pi}, the quotient
matrices of \(M(R_r(G))\), where \(M\in\mathcal{M}\), are as follows:
\[
        \mathcal Q_A(R_{r}(G))
        =
        \begin{pmatrix}
        A_G & P\\
        I   & 0
        \end{pmatrix},
        \qquad
        \mathcal Q_L(R_{r}(G))
        =
        \begin{pmatrix}
        rI-A_G & -P\\
        -I     & I
        \end{pmatrix},
\]
\[
        \mathcal Q_Q(R_{r}(G))
        =
        \begin{pmatrix}
        rI+A_G & P\\
        I      & I
        \end{pmatrix},
        \qquad
        \mathcal Q_D(R_{r}(G))
        =
        \begin{pmatrix}
        D_G       & (D_G+J)P\\
        D_G+J     & (D_G+2J)P-2I
        \end{pmatrix},
\]
and for \(\varepsilon\in\{1,-1\}\),
\[
        \mathcal Q_{\varepsilon}^{\deg}(R_{r}(G))
        =
        \begin{pmatrix}
        rI+\varepsilon D_G
            & \varepsilon(D_G+J)P\\
        \varepsilon(D_G+J)
            & I+\varepsilon\bigl((D_G+2J)P-2I\bigr)
        \end{pmatrix}.
\]
 Moreover,
\(\mathcal Q_{1}^{\deg}(R_r(G))\) and
\(\mathcal Q_{-1}^{\deg}(R_r(G))\) are the quotient matrices of
\(\Ddegp(R_r(G))\) and \(\Ddeg(R_r(G))\), respectively.
\end{proposition}

\begin{proof}

We prove the statement by explicitly analyzing the row structure of each matrix
with respect to the equitable partition
\[
\Pi = \{O_1,\ldots,O_n\} \cup \{L_i : i \in I_+\},
\]
where $O_i=\{v_i\}$ and $L_i$ is the set of leaves attached to $v_i$.

For adjacency matrix $A$, an original vertex $v_i$, its neighbors inside $R_r(G)$ consist of the neighbors of $v_i$ in $G$, and all leaves in $L_i$. Hence the row block corresponding to $(O_i, O_j)$ is $(A_G)_{ij}$,
and $(O_i, L_j)$ is $P_{ij}$. For a leaf vertex in $L_i$, it is adjacent only to $v_i$,
which yields the block $(I,0)$. Thus,
\[
\mathcal Q_A(R_r(G)) =
\begin{pmatrix}
A_G & P \\
I   & 0
\end{pmatrix}.
\]

For $L = \Deg - A$ and $Q = \Deg + A$, since every original vertex has degree $r$ in $R_r(G)$, we have
\[
\Deg(R_r(G)) =
\begin{pmatrix}
rI & 0 \\
0  & I
\end{pmatrix}.
\]
Then we immediately have
\[
\mathcal Q_L(R_r(G)) =
\begin{pmatrix}
rI - A_G & -P \\
-I & I
\end{pmatrix},
\quad
\mathcal Q_Q(R_r(G)) =
\begin{pmatrix}
rI + A_G & P \\
I & I
\end{pmatrix}.
\]

For Distance matrix $D$, we analyze row sums according to vertex type. For $v_i \in O_i$,
every leaf $u \in L_j$ satisfies
\[
\mathrm{dist}(v_i,u) = \mathrm{dist}_G(v_i,v_j) + 1,
\] and hence the
corresponding block is \((D_G+J)P\). For a leaf $u \in L_i$, the distances to the
original vertices are given by the \(i\)-th row of \(D_G+J\). Moreover, its total
distance to the leaves in \(L_j\) is \(p_j(\dist_G(v_i,v_j)+2)\) when \(j\ne i\),
whereas its total distance to the other leaves in \(L_i\) is \(2(p_i-1)\). This
gives the lower-right block \((D_G+2J)P-2I\).
Thus,
\[
\mathcal Q_D(R_r(G)) =
\begin{pmatrix}
D_G & (D_G + J)P \\
D_G + J & (D_G + 2J)P - 2I
\end{pmatrix}.
\]

Recall $\Ddegp = \Deg + D$ and $ \Ddeg = \Deg - D.$ Then
we obtain directly
\[
\mathcal Q^{\deg}_{\varepsilon}(R_r(G)) =
\begin{pmatrix}
rI + \varepsilon D_G & \varepsilon(D_G + J)P \\
\varepsilon(D_G + J) & I + \varepsilon((D_G + 2J)P - 2I)
\end{pmatrix},
\quad \varepsilon \in \{1,-1\}.
\]

 The proof has been completed.
\end{proof}

\begin{remark}
If \(I_+\ne\{1,\ldots,n\}\), then each of the quotient
matrices above is obtained from the corresponding displayed
block quotient by deleting the row and column indexed by \(L_i\), for each
\(i\notin I_+\).
\end{remark}

\begin{lemma}\label{lem:quotient-factorization}
Let \(I_+=\{i:p_i>0\}\), and $ m(R_r(G)):=\sum_{i\in I_+}(p_i-1)$.
With respect to the equitable partition \(\Pi\) in \eqref{pi}, the
characteristic polynomials of the matrices associated with \(R_r(G)\) satisfy
\begin{align*}
        \chi_{A(R_r(G))}(x)
        &=x^{m(R_r(G))}\chi_{\mathcal Q_A(R_r(G))}(x),\\
        \chi_{L(R_r(G))}(x)
        &=(x-1)^{m(R_r(G))}\chi_{\mathcal Q_L(R_r(G))}(x),\\
        \chi_{Q(R_r(G))}(x)
        &=(x-1)^{m(R_r(G))}\chi_{\mathcal Q_Q(R_r(G))}(x),\\
        \chi_{D(R_r(G))}(x)
        &=(x+2)^{m(R_r(G))}\chi_{\mathcal Q_D(R_r(G))}(x).
\end{align*}
Moreover, for each \(\varepsilon\in\{1,-1\}\),
\[
        \chi_{\Deg(R_r(G))+\varepsilon D(R_r(G))}(x)
        =
        \bigl(x-(1-2\varepsilon)\bigr)^{m(R_r(G))}
        \chi_{\mathcal Q_{\varepsilon}^{\deg}(R_r(G))}(x).
\]
\end{lemma}

\begin{proof}
Let \(S\) be the characteristic matrix of \(\Pi\), and
let \(\mathcal U=\operatorname{Col}(S)\) denote its column space. By
Proposition~\ref{prop:quotient-matrices}, for each matrix $M\in\mathcal{M}$, then we have
\[
        M(R_r(G))S
        =
        S\mathcal Q_M(R_r(G)),
\]
where \(\mathcal Q_M(R_r(G))\) denotes the quotient matrix associated with $M$.  For details, see \cite[Section 2.3]{Brouwer}.
Consequently, \(\mathcal U\) is invariant under every matrix under
consideration. Moreover, with respect to the basis formed by the columns of
\(S\), the restriction of each matrix to \(\mathcal U\) is represented by
the corresponding quotient matrix.

Define
\[
\mathcal W
=
\left\{
 z\in\RR^{V(R_r(G))}:
 z_{v_i}=0\ \text{for }1\le i\le n,
 \quad
 \sum_{u\in L_i}z_u=0\ \text{for every }i\in I_+
\right\}.
\]
Then
\[
\dim \mathcal{W}
=
\sum_{i\in I_+}(p_i-1)
=
m(R_r(G)).
\]
Furthermore,
\[
       \dim\mathcal U+\dim \mathcal{W} |\Pi|+\dim\mathcal W
        =
        n+|I_+|+\sum_{i\in I_+}(p_i-1)
        =
        n+\sum_{i\in I_+}p_i
        =
        |V(R_r(G))|.
\]
The subspace \(\mathcal U\) consists precisely of the vectors that are constant
on each cell of \(\Pi\). Therefore, we obtain \(\mathcal W=\mathcal U^{\perp}\), and hence
\[
        \RR^{V(R_r(G))}=\mathcal U\oplus\mathcal W.
\]

We now compute the action of the relevant matrices on \(\mathcal W\). Let
\(z\in\mathcal W\).  For the adjacency matrix, if \(v_i\) is an original
vertex, then
\[
        (Az)_{v_i}
        =
        \sum_{u\sim v_i}z_u
        =
        \sum_{u\in L_i}z_u
        =
        0.
\]
 If \(u\in L_i\), then \(u\) is adjacent
only to \(v_i\), and hence
\[
        (Az)_u=z_{v_i}=0.
\]
Thus $Az=0$, and hence $z$ is an eigenvector of $A$ corresponding to eigenvalue $0$.

Moreover, $
        Lz=(\Deg-A)z=z,$ and $
        Qz=(\Deg+A)z=z.$ Thus, $z$ is an eigenvector of $Q$ and $L$ corresponding to $1$.

Next consider the distance matrix \(D=D(R_r(G))\). If \(v_s\) is an original
vertex, then the distance from \(v_s\) to a vertex \(u\in L_i\) is independent
of the choice of \(u\) inside \(L_i\). Therefore
\[
\begin{aligned}
        (Dz)_{v_s}
        &=
        \sum_{i\in I_+}\sum_{u\in L_i}
        \dist(v_s,u)z_u  =
        \sum_{i\in I_+}
        \dist(v_s,v_{i,1})
        \sum_{u\in L_i}z_u
        =
        0.
\end{aligned}
\]
Now let \(u\in L_i\). For vertices \(w\in L_i\) with \(w\ne u\), we have
\(\dist(u,w)=2\). For each \(j\ne i\), the distance from \(u\) to \(w\in L_j\)
is independent of the choice of \(w\) inside \(L_j\). Hence
\[
\begin{aligned}
        (Dz)_u
        &=
        \sum_{\substack{w\in L_i\\ w\ne u}}
        \dist(u,w)z_w
        +
        \sum_{\substack{j\in I_+\\ j\ne i}}
        \sum_{w\in L_j}
        \dist(u,w)z_w  \\
        &=
        2\sum_{\substack{w\in L_i\\ w\ne u}}z_w
        +
        \sum_{\substack{j\in I_+\\ j\ne i}}
        \dist(u,v_{j,1})
        \sum_{w\in L_j}z_w
        =
        -2z_u.
\end{aligned}
\]
Thus, we obtain $Dz=-2z$, and hence $z$ is an eigenvector of $D$ corresponding to the eigenvalue $-2$.
Since \(\Deg z=z\) on \(\mathcal W\), it follows that
\[
        (\Deg+\varepsilon D)z
        =
        z+\varepsilon(-2z)
        =
        (1-2\varepsilon)z,
        \qquad \varepsilon\in\{1,-1\}.
\]
Then $1-2\varepsilon$ is eigenvalue of $\Deg+\varepsilon D$.
Each of these eigenvalues occurs with multiplicity at least
\(\dim\mathcal W=m(R_r(G))\).

 Taking characteristic polynomials yields the stated factorizations.
\end{proof}

\begin{remark}
If some \(p_i=0\), the same statement holds after deleting the empty cells
\(L_i\).  This is the case needed below when \(r=3\).  In that case the
quotient matrices are obtained from the same row-sum formulas, but only
nonempty leaf cells are included.
\end{remark}

\section{Proof of Theorem~\ref{thm:intro-main}}
\label{sec:main-proof}

This section is devoted to proving Theorem~\ref{thm:intro-main}.
In particular, the infinite families of pairs of trees constructed therein fail to be cospectral with respect to both $\Atrs$ and $\Atrsp$. We begin by presenting the explicit construction of this family of trees.

Let $T_1$ and $T_2$ be a pair of trees of order $17$, as illustrated in Figure~\ref{fig:cospectral graphs 1}.\\
\begin{figure}[ht]
\centering
  \begin{minipage}[t]{0.48\textwidth}
\centering
 \begin{tikzpicture}[scale=0.65, every node/.style={circle, fill=black, inner sep=1.6pt}, every edge/.style={draw=black, line width=2.8pt}]
 \footnotesize{

  \node (b1-1)  at (0,0) {};
    \node (b1-2)  at (0.7,0.6) {};
    \node (b1-3)  at (-0.7,0.6) {};

  \node (b1-4)  at (-1.5,1.4) {};
    \node (b1-5)  at (-2.0,2.2) {};
    \node (b1-6)  at (-2.5,3.0) {};
    \node (b1-7)  at (-1.2,2.2) {};
    \node (b1-8)  at (-0.9,3.0) {};
   \node (b1-9)  at (-0.4,1.4) {};

    \node (b1-10) at (1.7,1.4) {};
    \node (b1-11) at (2.1,2.2) {};
    \node (b1-12) at (2.5,3.0) {};
    \node (b1-13) at (2.1,4.0) {};
    \node (b1-14) at (3.0,4.0) {};
    \node (b1-15) at (0.8,1.4) {};
    \node (b1-16) at (0.5,2.2) {};
    \node (b1-17) at (0.2,3.0) {};

        \draw [line width=0.8](b1-1) -- (b1-2);
        \draw [line width=0.8](b1-1) -- (b1-3);
         \draw [line width=0.8](b1-3) -- (b1-9);
         \draw [line width=0.8](b1-3) -- (b1-4);
         \draw [line width=0.8](b1-4) -- (b1-5);
         \draw [line width=0.8](b1-5) -- (b1-6);
         \draw [line width=0.8](b1-4) -- (b1-7);
         \draw [line width=0.8](b1-7) -- (b1-8);
         \draw [line width=0.8](b1-2) -- (b1-15);
         \draw [line width=0.8](b1-15) -- (b1-16);
         \draw [line width=0.8](b1-16) -- (b1-17);
         \draw [line width=0.8](b1-2) -- (b1-10);
         \draw [line width=0.8](b1-10) -- (b1-11);
         \draw [line width=0.8](b1-11) -- (b1-12);
         \draw [line width=0.8](b1-12) -- (b1-13);
         \draw [line width=0.8](b1-12) -- (b1-14);

 }

     \end{tikzpicture}
\end{minipage}
\begin{minipage}[t]{0.48\textwidth}
\centering
 \begin{tikzpicture}[scale=0.65, every node/.style={circle, fill=black, inner sep=1.6pt}, every edge/.style={draw=black, line width=2.8pt}]
 \footnotesize{

  \node (b1-1)  at (0,0) {};
    \node (b1-2)  at (0.7,0.6) {};
    \node (b1-3)  at (-0.7,0.6) {};

  \node (b1-4)  at (-1.5,1.4) {};
    \node (b1-5)  at (-2.0,2.2) {};
    \node (b1-6)  at (-2.5,3.0) {};
    \node (b1-7)  at (-1.2,2.2) {};
    \node (b1-8)  at (-0.9,3.0) {};

    \node (b1-9)  at (1.3,2.2) {};
    \node (b1-10) at (1.7,1.4) {};
    \node (b1-11) at (2.1,2.2) {};
    \node (b1-12) at (2.5,3.0) {};
    \node (b1-13) at (2.1,4.0) {};
    \node (b1-14) at (3.0,4.0) {};

    \node (b1-15) at (0.5,1.4) {};
    \node (b1-16) at (0.2,2.2) {};
    \node (b1-17) at (-0.1,3.0) {};

        \draw [line width=0.8](b1-1) -- (b1-2);
        \draw [line width=0.8](b1-1) -- (b1-3);
         \draw [line width=0.8](b1-10) -- (b1-9);
         \draw [line width=0.8](b1-3) -- (b1-4);
         \draw [line width=0.8](b1-4) -- (b1-5);
         \draw [line width=0.8](b1-5) -- (b1-6);
         \draw [line width=0.8](b1-4) -- (b1-7);
         \draw [line width=0.8](b1-7) -- (b1-8);
         \draw [line width=0.8](b1-2) -- (b1-15);
         \draw [line width=0.8](b1-15) -- (b1-16);
         \draw [line width=0.8](b1-16) -- (b1-17);
         \draw [line width=0.8](b1-2) -- (b1-10);
         \draw [line width=0.8](b1-10) -- (b1-11);
         \draw [line width=0.8](b1-11) -- (b1-12);
         \draw [line width=0.8](b1-12) -- (b1-13);
         \draw [line width=0.8](b1-12) -- (b1-14);

 }

     \end{tikzpicture}
\end{minipage}

\caption{\label{fig:cospectral graphs 1} A pair of trees $T_1$ (left) and $T_2$ (right). }
\end{figure}
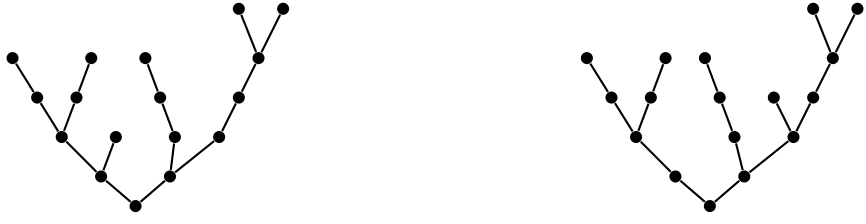
\begin{lemma}\label{lem:B1B2-nonisomorphic}
The pair of trees \(T_1\) and \(T_2\) given as in Fig.~\ref{fig:cospectral graphs 1} are non-isomorphic.
\end{lemma}

\begin{proof}
Consider the vertices of degree 3 in $T_1$ and $T_2$, respectively.
In $T_1$, the multiset of pairwise distances among them is
\[
        \{1,2,3,3,5,6\}.
\]
In $T_2$, the multiset of pairwise distances among them is
\[
        \{1,2,3,3,4,6\}.
\]
This multiset is an isomorphism invariant. As the two multisets are different, we conclude that \(T_1\not\cong T_2\).
\end{proof}

The maximum degrees of $T_1$ and $T_2$ are both equal to \(3\).
For each integer \(r\ge 3\), let $R_r(T_1)$ and $R_r(T_2)$ be \(r\)-regularized leaf extension of \(T_1\) and \(T_2\), respectively. We define $ T_1^{(r)}=R_r(T_1)$, and
        $T_2^{(r)}=R_r(T_2).$

\begin{lemma}\label{lem:T-nonisomorphic}

For every \(r\ge 3\), the trees \(T_1^{(r)}\) and \(T_2^{(r)}\) are non-isomorphic trees of order \(17r-15\).
\end{lemma}

\begin{proof}
Since \(T_1\) and \(T_2\) have \(17\) vertices and \(16\) edges, the
number of leaves attached to the 17 vertices is
\[
        \sum_{v\in V(T_k)}(r-\deg_{T_k}(v))
        =
        17r-2\cdot 16
        =
        17r-32,
\]
and hence
\[
        |V(T_k^{(r)})|
        =
        17+(17r-32)
        =
        17r-15,
\]
for $k=1,2$.
In \(T_k^{(r)}\), the original vertices are precisely the vertices of degree \(r\), while all newly attached vertices are leaves. Hence, any isomorphism between \(T_1^{(r)}\) and \(T_2^{(r)}\) restrict to an isomorphism between the subgraphs induced by their degree \(r\) vertices, namely between \(T_1\) and \(T_2\). This contradicts Lemma~\ref{lem:B1B2-nonisomorphic}.
\end{proof}


\begin{lemma}\label{prop:quotient-identities}

Let $M \in \mathcal{M}$ be a matrix type.  For \(r\ge 3\), the quotient characteristic polynomials associated with $T_1^{(r)}$ and $T_2^{(r)}$
are equal.
\end{lemma}

\begin{proof}
Firstly, we distinguish the following two cases:

\noindent\textbf{Case 1.} $r = 3$.
After removing the empty leaf cells, direct computation shows that the
characteristic polynomials of the two quotient matrices corresponding to $M$ for $T_1^{(r)}$ and $T_2^{(r)}$ are identical, where
$M \in  \mathcal{M}$.

\noindent\textbf{Case 2.} $r \ge 4$.
 Each leaf cell is nonempty, so all associated quotient matrices have order exactly $34$.  For a fixed matrix type \(M\), let
\[
        \Delta_M(r,x)
        =
        \chi_{\mathcal Q_M(T_1^{(r)})}(x)
        -
        \chi_{\mathcal Q_M(T_2^{(r)})}(x),
\]
where \(\mathcal Q_M(T_k^{(r)})\) denotes the quotient matrix corresponding to $M$ on $T_k^{(r)}$, for $k=1,2$.

By Proposition~\ref{prop:quotient-matrices}, we know that each entry of the quotient matrices is affine in \(r\). Consequently, $\Delta_M(r,x) \in \mathbb{Z}[r,x]$ can be treated as a polynomial of degree at most $34$ in $r$, whose each coefficient is a polynomial in \(x\). Thus we can write
\[
    \Delta_M(r,x) = \sum_{k=0}^{34} c_k(x)\, r^k,
\]
where $c_k(x) \in \ZZ[x]$.
A direct exact computation using Python shows that
\[
    \Delta_M(s,x) = 0,
    \qquad\text{for } s = 4, 5, \dots, 38.
\]
 It indicates that $\Delta_M(s,x)=0$ has $35$ distinct integer solutions,  so every coefficient polynomial $c_k(x)$ satisfies $c_k(x) = 0$ for $0 \le k \le 34$.  Thus, $\Delta_M(r,x) = 0$ identically over $\mathbb{Z}[r,x]$, and  hence the two
quotient characteristic polynomials coincide for every matrix type $M$. Calculations were performed using Python; for details, refer to the appendix~\ref{app:python}.

 Combining the two cases completes the proof.
\end{proof}
 We now prove Theorem~\ref{thm:intro-main} by combining the preceding results.
\begin{proof}[Proof of Theorem~\ref{thm:intro-main}]
 Lemma~\ref{lem:T-nonisomorphic} gives that \(T_1^{(r)}\) and \(T_2^{(r)}\) are non-isomorphic trees of order \(17r-15\), for every \(r\ge 3\).
It remains to prove cospectrality.  Since \(T_1\) and \(T_2\) have the same
degree sequence, the multisets $ \{r-\deg_{T_1}(v):v\in V(T_1)\}$ and $\{r-\deg_{T_2}(v):v\in V(T_2)\}$ are equal.  Hence, \(m(T_1^{(r)})=m(T_2^{(r)})\) in
Lemma~\ref{lem:quotient-factorization}.

These quotient characteristic polynomials agree
by Lemma~\ref{prop:quotient-identities}. Therefore, from the factorization of the characteristic polynomial given in Lemma~\ref{lem:quotient-factorization}, we obtain
\(T_1^{(r)}\) and \(T_2^{(r)}\) are simultaneously cospectral for $\{A, L, Q, D, \Ddegp, \Ddeg\}$.
\end{proof}

\begin{corollary}\label{cor:counterexamples}
Not all trees are determined by the \(\Ddegp\)-spectrum and are not determined by
the \(\Ddeg\)-spectrum.
\end{corollary}
\begin{proof}
        By Theorem~\ref{thm:intro-main}, we gives an infinite family of pairs of non-isomorphic
\(\Ddegp\)-cospectral trees and  \(\Ddeg\)-cospectral trees,
so not all trees are determined by either of these spectrum.
\end{proof}

\begin{corollary}\label{cor:tuple-spectra}

The spectrum of matrices chosen from $\mathcal{M}$
cannot uniquely determines all trees.
\end{corollary}

\begin{proof}
The pair of trees \(T_1^{(r)}\) and \(T_2^{(r)}\) have the same spectrum for each
matrix in the displayed list, but are non-isomorphic.
\end{proof}

 It is therefore natural to ask whether the same
phenomenon persists for other matrices. The next result shows
that this is not the case for the transmission-adjacency matrix. In particular,
the above pair of trees is separated by the \(\Atrs\)-spectrum, and hence also
by the \(\Atrsp\)-spectrum.
\begin{lemma}\label{prop:family-not-Atrs-cospectral}
For every integer \(r\ge 3\), the trees \(T_1^{(r)}\) and \(T_2^{(r)}\) are
not \(\Atrs\)-cospectral. Consequently, they are not \(\Atrsp\)-cospectral.
\end{lemma}

\begin{proof}
We first record a trace identity. Let \(T\) be a tree, let
\(R=\Trs(T)\), and let \(A=A(T)\). Since $\Atrs(T)=R-A$,
we have
\[
\begin{aligned}
\operatorname{tr}\bigl(\Atrs(T)^3\bigr)
&=
\operatorname{tr}\bigl((R-A)^3\bigr) \\
&=
\operatorname{tr}(R^3)
-\operatorname{tr}(R^2A)
-\operatorname{tr}(RAR)
-\operatorname{tr}(AR^2)  \\
&\quad
+\operatorname{tr}(RA^2)
+\operatorname{tr}(ARA)
+\operatorname{tr}(A^2R)
-\operatorname{tr}(A^3).
\end{aligned}
\]
The three terms containing exactly one factor \(A\) have trace zero, as
\(R\) is diagonal and \(A\) has zero diagonal. Also,
 since \(T\) is triangle-free, we have \(\operatorname{tr}(A^3)=0\).
For a vertex \(v_i\in V(T)\), since $A$ is a symmetric
$\{0,1\}$
-matrix, then we have
\[(A^2)_{ii}
        =
        \sum_{j=1}^{|V(T)|} a_{ij}a_{ji}
        =
        \sum_{j=1}^{|V(T)|} a_{ij}^2
        = \sum_{j=1}^{|V(T)|} a_{ij}
        =
        \deg_T(v_i).
\]
Since \(R\) is diagonal with \(R_{ii}=\trs_T(v_i)\), it follows that
\[
         \operatorname{tr}(A^2R)
        =\sum_{i=1}^{|V(T)|}(A^2R)_{ii}
        =
        \sum_{i=1}^{|V(T)|}(A^2)_{ii}R_{ii}
        =
        \sum_{v_i\in V(T)}\deg_T(v_i)\trs_T(v_i).
\]
Similarly, we obtain that $   \operatorname{tr}(A^2R)=\operatorname{tr}(RA^2)
        =
        \operatorname{tr}(ARA)
       .$
Therefore,
\[
\begin{aligned}
   \operatorname{tr}\bigl(\Atrs(T)^3\bigr)&=
\operatorname{tr}(R^3)
+
\operatorname{tr}(RA^2)
+
\operatorname{tr}(ARA)
+
\operatorname{tr}(A^2R)\\
&=
\sum_{v_i\in V(T)} \trs_T(v_i)^3
+
3\sum_{v_i\in V(T)} \trs_T(v_i)\deg_T(v_i).
\end{aligned}
\]

Now fix \(k\in\{1,2\}\), and write
$  d_{ij}^{(k)}=\dist_{T_k}(i,j),$ and $
        p_i^{(k)}=r-\deg_{T_k}(i).$
Then \(p_i^{(k)}\) is the number of new leaves attached to the original vertex
\(i\) in  \(T_k^{(r)}\). In \(T_k^{(r)}\), the transmission of the original vertex \(i\) is
\[
        \tau_i^{(k)}
        =
        \sum_{j=1}^{17} d_{ij}^{(k)}
        +
        \sum_{j=1}^{17}
        p_j^{(k)}\bigl(d_{ij}^{(k)}+1\bigr).
\]
If \(p_i^{(k)}>0\), then the transmission of any leaf attached to \(i\) is
\[
        \sigma_i^{(k)}
        =
        \sum_{j=1}^{17}\bigl(d_{ij}^{(k)}+1\bigr)
        +
        2\bigl(p_i^{(k)}-1\bigr)
        +
        \sum_{\substack{1\le j\le 17\\ j\ne i}}
        p_j^{(k)}\bigl(d_{ij}^{(k)}+2\bigr).
\]
If \(p_i^{(k)}=0\), no leaf is attached to \(i\). In this case the
corresponding terms below are multiplied by \(p_i^{(k)}\), and hence make no
contribution.

Since the original vertices have degree \(r\) in \(T_k^{(r)}\), whereas all
new vertices are leaves, the preceding trace identity gives

\begin{align}\label{main eq.}
\operatorname{tr}\bigl(\Atrs(T_k^{(r)})^3\bigr)
&=
\sum_{i=1}^{17} \bigl(\tau_i^{(k)}\bigr)^3
+
\sum_{i=1}^{17}
p_i^{(k)}\bigl(\sigma_i^{(k)}\bigr)^3  +
3\left(
r\sum_{i=1}^{17}\tau_i^{(k)}
+
\sum_{i=1}^{17}p_i^{(k)}\sigma_i^{(k)}
\right).
\end{align}

Using the explicit edge sets of \(T_1\) and \(T_2\), the values of
\(\tau_i^{(k)}\) and \(\sigma_i^{(k)}\) can be computed. Substituting these values into Eq.~(\ref{main eq.})
yields
\[
\operatorname{tr}\bigl(\Atrs(T_1^{(r)})^3\bigr)
=
19486776r^4
-
78568899r^3
+
118786167r^2
-
79812316r
+
20108280,
\]
and
\[
\operatorname{tr}\bigl(\Atrs(T_2^{(r)})^3\bigr)
=
19485720r^4
-
78564675r^3
+
118779831r^2
-
79808092r
+
20107224.
\]
Therefore
\[
\begin{aligned}
&\operatorname{tr}\bigl(\Atrs(T_1^{(r)})^3\bigr)
-
\operatorname{tr}\bigl(\Atrs(T_2^{(r)})^3\bigr)  \\
&\qquad =
1056r^4
-
4224r^3
+
6336r^2
-
4224r
+
1056  \\
&\qquad =
1056(r-1)^4.
\end{aligned}
\]
This number is nonzero for every \(r\ge 3\). Hence the third spectral moments
of \(\Atrs(T_1^{(r)})\) and \(\Atrs(T_2^{(r)})\) are different. Therefore, the
two trees are not \(\Atrs\)-cospectral.

Since tree is bipartite, Corollary~\ref{cor:trees-plus-minus-Atrs}
implies that the \(\Atrsp\)-spectrum and the \(\Atrs\)-spectrum coincide for
each tree. Thus, if \(T_1^{(r)}\) and \(T_2^{(r)}\) are not
\(\Atrs\)-cospectral, then they also are not \(\Atrsp\)-cospectral.
\end{proof}

\section*{Acknowledgements}

The Python/SymPy codes provided in Appendix A were generated with assistance from \mbox{GPT-5.6-sol}. All computational outputs have been independently verified by the authors. The authors bear full responsibility for the accuracy of the manuscript.

\appendix
\section{Python/SymPy code for the computational verification}
\label{app:python}

The following Python script verifies Lemma~\ref{prop:quotient-identities}.
It uses exact integer arithmetic only.  The check with \(r=3\) is performed
over \(\ZZ[x]\), with empty leaf cells deleted.  For \(r\ge4\), the quotient
matrices have order \(34\) and every entry is affine in \(r\).  Hence the
difference of the two quotient characteristic polynomials has degree at most
\(34\) in \(r\), coefficientwise in \(x\).  The script therefore checks the
identities exactly for \(r=4,5,\ldots,38\); these \(35\) values prove the
identities in \(\ZZ[r,x]\).

\begin{verbatim}
from collections import deque
import sympy as sp

E1 = [(1,2),(1,3),(2,10),(2,15),(3,4),(3,9),(4,5),(4,7),
      (5,6),(7,8),(10,11),(11,12),(12,13),(12,14),
      (15,16),(16,17)]

E2 = [(1,2),(1,3),(2,9),(2,15),(3,4),(4,5),(4,7),(5,6),
      (7,8),(9,10),(9,14),(10,11),(11,12),(11,13),
      (15,16),(16,17)]

N = 17
KINDS = ["A", "L", "Q", "D", "Ddeg+", "Ddeg-"]
x = sp.symbols("x")

def graph_data(edges):
    adj = [[] for _ in range(N)]
    A = sp.zeros(N, N)
    deg = [0] * N

    for a, b in edges:
        a -= 1
        b -= 1
        A[a, b] = 1
        A[b, a] = 1
        adj[a].append(b)
        adj[b].append(a)
        deg[a] += 1
        deg[b] += 1

    D = sp.zeros(N, N)
    for s in range(N):
        dist = [-1] * N
        dist[s] = 0
        queue = deque([s])
        while queue:
            u = queue.popleft()
            for v in adj[u]:
                if dist[v] == -1:
                    dist[v] = dist[u] + 1
                    queue.append(v)
        for j, d in enumerate(dist):
            D[s, j] = d

    return A, D, deg

DATA1 = graph_data(E1)
DATA2 = graph_data(E2)

def quotient(data, r, kind, include_zero_leaf_cells):
    A, D, deg = data
    r = sp.Integer(r)
    p = [r - deg[i] for i in range(N)]

    cells = [("O", i) for i in range(N)]
    for i in range(N):
        if include_zero_leaf_cells or p[i] != 0:
            cells.append(("L", i))

    Q = sp.zeros(len(cells), len(cells))

    for row, (ta, i) in enumerate(cells):
        for col, (tb, j) in enumerate(cells):

            if kind == "A":
                if ta == "O" and tb == "O":
                    Q[row, col] = A[i, j]
                elif ta == "O" and tb == "L":
                    Q[row, col] = p[j] if i == j else 0
                elif ta == "L" and tb == "O":
                    Q[row, col] = 1 if i == j else 0

            elif kind in ["L", "Q"]:
                sgn = -1 if kind == "L" else 1
                if ta == "O" and tb == "O":
                    Q[row, col] = r if i == j else sgn * A[i, j]
                elif ta == "O" and tb == "L":
                    Q[row, col] = sgn * p[j] if i == j else 0
                elif ta == "L" and tb == "O":
                    Q[row, col] = sgn if i == j else 0
                else:
                    Q[row, col] = 1 if i == j else 0

            elif kind == "D":
                if ta == "O" and tb == "O":
                    Q[row, col] = D[i, j]
                elif ta == "O" and tb == "L":
                    Q[row, col] = p[j] * (D[i, j] + 1)
                elif ta == "L" and tb == "O":
                    Q[row, col] = D[i, j] + 1
                else:
                    if i == j:
                        Q[row, col] = 2 * (p[i] - 1)
                    else:
                        Q[row, col] = p[j] * (D[i, j] + 2)

            elif kind in ["Ddeg+", "Ddeg-"]:
                eps = 1 if kind == "Ddeg+" else -1
                if ta == "O" and tb == "O":
                    Q[row, col] = r if i == j else eps * D[i, j]
                elif ta == "O" and tb == "L":
                    Q[row, col] = eps * p[j] * (D[i, j] + 1)
                elif ta == "L" and tb == "O":
                    Q[row, col] = eps * (D[i, j] + 1)
                else:
                    if i == j:
                        Q[row, col] = 1 + eps * 2 * (p[i] - 1)
                    else:
                        Q[row, col] = eps * p[j] * (D[i, j] + 2)

            else:
                raise ValueError(kind)

    return Q

def charpoly(data, r, kind, include_zero_leaf_cells):
    Q = quotient(data, r, kind, include_zero_leaf_cells)
    return sp.Poly(Q.charpoly(x).as_expr(), x)

# r = 3: some leaf cells are empty and must be deleted.
for kind in KINDS:
    c1 = charpoly(DATA1, 3, kind, include_zero_leaf_cells=False)
    c2 = charpoly(DATA2, 3, kind, include_zero_leaf_cells=False)
    assert c1 == c2, f"failed at r=3, kind={kind}"
    if c1 == c2:
        print(f"  r=3, kind={kind}: PASS")

# For r >= 4 all leaf cells are nonempty.  The quotient matrices have
# order 34, and all entries are affine in r.  Thus the difference of the
# two quotient characteristic polynomials has degree at most 34 in r,
# coefficientwise in x.  Checking 35 distinct integer values proves the
# identity in Z[r,x].
for kind in KINDS:
    for r in range(4, 39):
        c1 = charpoly(DATA1, r, kind, include_zero_leaf_cells=True)
        c2 = charpoly(DATA2, r, kind, include_zero_leaf_cells=True)
        assert c1 == c2, f"failed at r={r}, kind={kind}"
        if c1 == c2:
            print(f"  r={r}, kind={kind}: PASS")
        else:
            print(f"  r={r}, kind={kind}: FAIL")
            raise AssertionError

print("All exact quotient determinant identities verified.")
print("For r >= 4, equality at r = 4,...,38 proves the identity in Z[r,x].")

\end{verbatim}


\begin{thebibliography}{99}


\bibitem{AbiadAlfaroVillagran2025}
A.~Abiad, C.~A. Alfaro and R.~R. Villagr\'an,
Distinguishing graphs by their spectra, Smith normal forms and complements,
\emph{Applied Mathematics and Computation} \textbf{490} (2025), 129198.

\bibitem{AlfaroZapata2024}
C.~A. Alfaro, O.~Zapata,
The degree-distance and transmission-adjacency matrices,
\emph{Computational and Applied Math. }\textbf{43} (2024), Article 351.

\bibitem{Brouwer} A. E. Brouwer, W. H. Haemers, \emph{Spectra of Graphs}, Springer, 2011.

\bibitem{CS} L. Collatz, U. Sinogowitz, Spektren endlicher Grafen, \emph{Abh. Math. Sem. Univ. Hamburg} \textbf{21} (1957)
63-77.

\bibitem{DH}
E.R. Van Dam, W. H. Haemers, Which graphs are determined by their spectrum?
\emph{Linear Algebra Appl}. \textbf{373} (2003) 241-272.

\bibitem{DH1} E. R. Van Dam, W.H. Haemers, Developments on spectral characterizations
of graphs, \emph{Discrete Math}. \textbf{309} (2009) 576-586.

\bibitem{W4} W. Wang, A simple arithmetic criterion for graphs being
determined by their generalized spectra, \emph{J. Combin. Theory, Ser. B} \textbf{122} (2017) 438-451.

\bibitem{W3} W. Wang, Generalized spectral characterization of graphs revisited, \emph{Electronic J.
Combin.} \textbf{20 (4)} (2013), \#P4.


\end{thebibliography}
\end{document}